\documentclass[a4paper]{amsart}

\usepackage{amsmath,amssymb}
\usepackage{mathrsfs}

\newtheorem{theorem}{Theorem}[section]
\newtheorem*{theorem*}{Theorem}
\newtheorem{lemma}[theorem]{Lemma}
\newtheorem{proposition}[theorem]{Proposition}
\newtheorem*{proposition*}{Proposition}
\newtheorem{corollary}[theorem]{Corollary}
\theoremstyle{definition}
\newtheorem{definition}[theorem]{Definition}
\newtheorem{example}[theorem]{Example}
\newtheorem{examples}[theorem]{Examples}

\newtheorem{remark}[theorem]{Remark}

\newcommand{\NN}{\mathbb{N}}

\newcommand{\RR}{\mathbb{R}}
\newcommand{\CC}{\mathbb{C}}

\newcommand{\holalg}[1]{{\mathcal{O}\left( {#1} \right)}}
\newcommand{\id}{{\mathrm{id}}}
\newcommand{\aut}[1]{{\mathrm{Aut}\left( {#1} \right)}}
\newcommand{\autvol}[2]{{\mathrm{Aut_{ #2 }}\left( {#1} \right)}}
\newcommand{\autid}[1]{{\mathrm{Aut}^*\!\left( {#1} \right)}}
\newcommand{\autvolid}[2]{{\mathrm{Aut_{ #2 }^*}\left( {#1} \right)}}
\newcommand{\cont}{\mathcal{C}}

\title{Free Dense Subgroups of Holomorphic Automorphisms}
\author{Rafael B. Andrist}
\author{Erlend Forn\ae ss Wold}

\begin{document}

\begin{abstract}
We show the existence of free dense subgroups, generated by 2 elements, in the holomorphic shear and overshear group of complex-Euklidean space and extend this result to the group of holomorphic automorphisms of Stein manifolds with Density Property, provided there exists a generalized translation. The conjugation operator associated to this generalized translation is hypercyclic on the topological space of holomorphic automorphisms.
\end{abstract}

\maketitle

\section{Introduction}

In Functional Analysis and Topological Dynamics, the phenomenon of so-called hypercyclicity of operators and universality of sequences of operators is studied. For a detailed survey, we refer to Grosse-Erdmann \cite{survey}. One of the early examples actually comes from Complex Analysis; in 1929 G.D. Birkhoff \cite{Birkhoff} constructed a holomorphic function $f : \CC \to \CC$ such that for a given sequence  $\{ a_m \}_{m = 1}^\infty \subset \RR$ with $\lim_{m \to \infty} a_m = \infty$ the set $\{z \mapsto f(z + a_m), \; m \in \NN \}$ is dense (in compact-open topology) in $\holalg{\CC}$. It follows in particular that for any translation $\tau : \CC \to \CC, \; \tau \not\equiv \id,$ the set $\{f \circ \tau^m : m \in \NN \}$ is dense in $\holalg{\CC}$. This motivated the following definition:
\begin{definition}
Let $E$ be a topological space.

A self-map $T : E \to E$ is called \emph{hypercyclic} if there exists an $f \in E$, called \emph{hypercyclic element} for $T$, such that the orbit $\{ T^m(f), \; m \in \NN \}$ is dense in $E$.

\end{definition}
The result of Birkhoff was later generalized to the case of holomorphic functions on $\CC^n$ and more recently by Zajac \cite{Zajac} also to holomorphic functions on pseudo-convex domains of $\CC^n$; he gives (in slightly more general form) the following Theorem which characterizes hypercyclic composition operators:
\begin{theorem}[Theorem 3.4. in \cite{Zajac}]
\label{thm:pseudoconv}
Let $\Omega \subset \CC^n$ by a pseudo-convex domain and let $\tau \in \holalg{\Omega, \Omega}$ be a holomorphic selfmap. By $C_\tau : \holalg{\Omega} \to \holalg{\Omega}$ we denote the composition operator $f \mapsto f \circ \tau$.
The operator $C_\tau$ is hypercyclic if and only if $\tau$ is injective and for every $\Omega$-convex compact subset $K \subset \Omega$ there exists $m \in \NN$ such that $\tau^m(K) \cap K = \emptyset$ and $K \cup \tau^m(K)$ is $\Omega$-convex.
\end{theorem}

We now want to consider groups of holomorphic automorphisms and introduce the following notation:

\begin{definition} \hfill
\begin{enumerate}
\item Let $X$ be a complex manifold. By $\aut{X} := \{f : X \to X \; \mbox{biholomorphic}\}$ we denote the \emph{group of holomorphic automorphisms of $X$}, to be understood as topological group with the compact-open topology; and by $\autid{X}$ its identity component.
\item Let $X$ be a complex manifold of dimension $n \in \NN$. A holomorphic $n$-form $\omega$ on $X$ which is everywhere non-degenerate is called a \emph{holomorphic volume form}. By $\autvol{X}{\omega} := \{f : X \to X \; \mbox{biholomorphic}  \; :  \; f^\ast(\omega) = \omega \}$ we denote the \emph{group of volume-preserving holomorphic automorphisms of $X$}, and by $\autvolid{X}{\omega}$ its identity component.
\end{enumerate}
\end{definition}

These topological groups are metrizable with the metric of uniform convergence. However, these metric spaces will not be Cauchy-complete, as one can see for example in case of $\aut{\CC^n}$ and a sequence of automorphisms $z \mapsto M_k \cdot z$, where $M_k, \; k \in \NN,$ is a complex matrix with $\det M_k \neq 0$ but $\lim_{k \to \infty} \det M_k = 0$. The phenomenon of hypercyclicity is usually studied on separable Fr\'echet spaces or at least on complete metric spaces.

The results concerning hypercyclity of compositions operators on $\holalg{\CC^{n-1}}$ can be applied directly to the shear and overshear groups of $\CC^n$. The importance of these groups lies in the fact that they are dense subgroups of $\aut{\CC^n}$ resp. $\autvol{\CC^n}{\omega}, \; \omega = dz_1 \wedge \dots \wedge dz_n,$ as shown by Anders\'en \cite{Andersen} and Anders\'en--Lempert \cite{AndersenLempert}.

A holomorphic overshear $F : \CC^n \to \CC^n$ has, after possibly a special complex-linear change of coordinates, the form
\[
F(z_1, \dots, z_n) = (z_1, z_2, \dots, z_{n-1}, \exp(f(z_1, \dots, z_{n-1})) \cdot z_n + g(z_1, \dots, z_{n-1}))
\]
where $f, g \in \holalg{\CC^{n-1}}$. It is called a shear if $f \equiv 0$; in this case $\det(d F) = 1$ which is equivalent to $f^\ast(\omega) = \omega$. The (over-)shear group is the group generated by all the (over-)shears.

Using Theorem \ref{thm:pseudoconv} for $\Omega = \CC^{n-1} \subseteq \CC^{n-1}$ with $\tau$ being a translation, and applying this to every holomorphic direction, one obtains immediately that there are $n+1$ (over-)shears which generate a dense subgroup of the (over-)shear group. However, in Theorem \ref{thm:twoauteukl} we show that it is actually possible to generate dense subgroups by only $2$ automorphisms of a special form, one of them being a translation. Moreover we are able to show that this dense subgroup is generated freely by those $2$ automorphisms.

The conditions for $\tau$ which had been used already for Theorem \ref{thm:pseudoconv}, motivate the following definition which we use for the formulation of our main theorem.

\begin{definition}
Let $X$ be a complex manifold and $\tau \in \aut{X}$. We call $\tau$ a \emph{generalized translation} if for any holomorphically convex compact $K \subsetneq X$ there exists an $m \in \NN$ such that for the iterate $\tau^m$ it holds that 
\begin{enumerate}
\item $\tau^m(K) \cap K = \emptyset$ and
\item $\tau^m(K) \cup K$ is $\holalg{X}$-convex.
\end{enumerate}
\end{definition}
An obvious example for $X = \CC^n$ is the translation $z \mapsto z + a, \; a \neq 0$. 

\begin{theorem*}[\ref{thm:twoautdens}]
Let $X$ be a Stein manifold with density property and assume there exists a generalized translation $\tau \in \autid{X}$. Then there exists an $F \in \autid{X}$ such that $\left< \tau, F  \right>$ is a free and dense (in compact-open topology) subgroup of $\autid{X}$.
\end{theorem*}

From the point of view of hypercyclicity we can also state
\begin{theorem*}[\ref{thm:densehyperc}]
Let $X$ be a Stein manifold with density property and assume there exists a generalized translation $\tau \in \autid{X}$. Then the associated conjugation operator $\widetilde{C}_\tau$, defined by $\widetilde{C}_\tau(f) = \tau \circ f \circ \tau^{-1}$ is hypercyclic on $\autid{X}$.
\end{theorem*}

The notion of Density Property and so-called Anders\'en--Lempert Theory is introduced in Section \ref{sec:densprop}. Manifolds with density property in particular include $\CC^n, n \geq 2$, certain homogeneous spaces and all linear algebraic groups except those with connected components $\CC$ or $(\CC^\ast)^n$. The density property of Stein manifolds implies that the automorphism group is large, in particular infinite dimensional and acts $m$-transitively for any $m \in \NN$ (Varolin \cite{Varolin2}).

\section{Complex Euklidean Space}

Shears as a special kind of holomorphic automorphisms of $\CC^n$ have been introduced by Rosay and Rudin \cite{RosayRudin}. For a more recent treatment we refer to the textbook of Forsterni\v{c} \cite{steinmanifolds}.

\begin{definition}
A map
\[
F(z_1, \dots, z_n) = (z_1, z_2, \dots, z_{n-1}, \exp(f(z_1, \dots, z_{n-1})) \cdot z_n + g(z_1, \dots, z_{n-1}))
\]
where $f, g \in \holalg{\CC^{n-1}}$,
is called an \emph{overshear in direction of the $n$-th coordinate axis} (or a \emph{shear in direction of the $n$-th coordinate axis} if $f \equiv 0$). For $A \in \mathrm{SL}(\CC^n)$ the conjugates $A^{-1} \circ F \circ A$ are called \emph{overshears} (or \emph{shears} if $f \equiv 0$). The group generated by the shears resp. overshears is called the \emph{shear group} resp. \emph{overshear group}.
\end{definition}

\begin{remark}
The group generated by (over-)shears in all coordinate directions coincides with teh (over)-shear group.
\end{remark}

For our purposes it will be convenient to introduce the following notion:

\begin{definition}
Set $I : \CC^n \to \CC^n, \; n \geq 2,$ to be 
\[
I(z_1, \dots, z_n) := (z_2, \dots, z_n, -(-1)^n z_1)
\]
A map $F : \CC^n \to \CC^n$ which can be written as $F = G \circ I$, where $F$ is an (over-)shear is called a \emph{twisted (over-)shear}.
\end{definition}

\begin{lemma}
The groups generated by (over)-shears and twisted (over-)shears coincide.
\end{lemma}
\begin{proof}
By definition, $I$ is a twisted shear, hence every (over-)shear $G$ can be written as $G = F \circ I^{-1}$ where $F = G \circ I$ is a twisted (over-)shear. It remains to show that $I$ is a composition of shears: it is sufficent to show that any transposition $t$ of coordinates with a sign change can be written as a composition of shears, for simplicity $t(z_1, \cdots, z_{n-1}, z_n) = (z_1, \cdots, z_n, -z_{n-1})$. Let
\begin{align*}
A(z_1, \dots, z_{n-1}, z_n) &:= (z_1, \dots, z_{n-1}, z_n + z_{n-1}) \\
B(z_1, \dots, z_{n-1}, z_n) &:= (z_1, \dots, z_{n-1} - z_n, z_n)
\end{align*}
Then $t = A^{-1} \circ B \circ A$.
\end{proof}

For the following application of hypercyclicity we will need a slightly different statement than in Theorem \ref{thm:pseudoconv}. Despite that the method of proof is essentially the same and follows the original idea of G.D. Birkhoff \cite{Birkhoff}, we include the proof for a lack of reference.

\begin{proposition}
\label{prop:doublehypercyc}
Let $X$ be a Stein manifold and let $\tau : X \to X$ be a generalized translation. Then there exists a pair of holomorphic functions $f, g \in \holalg{X}$ such that for any pair of holomorphic functions $\hat{f}, \hat{g} \in \holalg{X}$ there is a subsequence $\{m_\ell\}_{\ell \in \NN} \subset \NN$ such that
\begin{align*}
 f \circ \tau^{m_{2\ell-1}} &\to \hat{f}, \\
 g \circ \tau^{m_{2\ell-1}} &\to 0, \\
 f \circ \tau^{m_{2\ell}}   &\to 0 \; \mbox{and} \\
 g \circ \tau^{m_{2\ell}}   &\to \hat{g}
\end{align*}
in compact-open topology.
\end{proposition}
\begin{proof}
Let $\{h_j\}_{j \in \NN} \subset \holalg{X}$ be a sequence such that $\{h_{2j}\}_{j \in \NN}$ and $\{h_{2j-1}\}_{j \in \NN}$ are both dense (in compact-open topology) in $\holalg{X}$. Because $X$ is Stein, there exists an exhaustion of $X$ with $\holalg{X}$-convex compacts $\{K_j\}_{j \in \NN}$. By definition of the generalized translation $\tau$, there exists an $m_j \in \NN$ such that $K_j \cap \tau^{m_j}(K_j) = \emptyset$ and $K_j \cup \tau^{m_j}(K_j)$ is $\holalg{X}$-convex. By omitting certain compacts of this exhaustion we may further assume that $K_j \cup \tau^{m_j}(K_j) \subset K_{j+1}$ after renumbering. Therefore there exist by Runge approximation two sequences $f_j, g_j \in \holalg{X}, j \in \NN,$ such that
\begin{align*}
 \sup_{\tau^{m_{2j-1}} (K_{2j-1})} d(f_{2j-1}, h_{2j-1} \circ \tau^{-m_{2j-1}}) < \frac{1}{2^{j}}, \quad &
 \sup_{K_{2j-1}} d(f_{2j-1}, 0) &< \frac{1}{2^{j}} \; \forall \, j \in \NN
 \\
 \sup_{\tau^{m_{2j-1}} (K_{2j-1})} d(f_{2j-1}, 0) &< \frac{1}{2^{2 \ell_j}}, \; \forall\, j < \ell_j
 \\ 
 \sup_{\tau^{m_{2j-1}} (K_{2j-1})} d(g_{2j-1}, 0) < \frac{1}{2^{j}}, \quad & 
 \sup_{K_{2j-1}} d(g_{2j-1}, 0) &< \frac{1}{2^{j}} \; \forall \, j \in \NN
 \\
 \sup_{\tau^{m_{2j-1}} (K_{2j-1})} d(g_{2j-1}, 0) &< \frac{1}{2^{2 \ell_j}}, \; \forall\, j < \ell_j
 \\ 
 \sup_{\tau^{m_{2j}} (K_{2j})} d(f_{2j}, 0) < \frac{1}{2^{j}}, \quad & 
 \sup_{K_{2j}} d(f_{2j}, 0) &< \frac{1}{2^{j}} \; \forall \, j \in \NN
 \\
 \sup_{\tau^{m_{2j}} (K_{2j})} d(f_{2j}, 0) &< \frac{1}{2^{2 \ell_j}}, \; \forall\, j < \ell_j
 \\ 
 \sup_{\tau^{m_{2j}} (K_{2j})} d(g_{2j}, h_{2j} \circ \tau^{-m_{2j}}) < \frac{1}{2^{j}}, \quad & 
 \sup_{K_{2j}} d(g_{2j}, 0) &< \frac{1}{2^{j}} \; \forall \, j \in \NN
 \\
 \sup_{\tau^{m_{2j}} (K_{2j})} d(g_{2j}, 0) &< \frac{1}{2^{2 \ell_j}}, \; \forall\, j < \ell_j
\end{align*}

Set $f := \sum_{j \in \NN} f_j$ and $g := \sum_{j \in \NN} g_j$. From the above estimates it is clear that the series $f$ and $g$ converge uniformly on compacts of $X$. Picking a subsequence of $\{h_j\}_{j \in \NN}$ such that $h_{2m_\ell-1} \to \hat{f}$ and $h_{2m_\ell} \to \hat{g}$ for $\ell \to \infty$ in compact-open topology, the following estimate holds uniformly on  $K_{j\ell}$: 
\begin{align*}
d(f \circ \tau^{m_{2 j_\ell - 1}}, \hat{f}) \leq 
&\sum_{j=1}^{2 j_\ell - 2} d(f_j \circ \tau^{ m_{2 j_\ell - 1}}, 0) \\
&+ d(f_{2 j_\ell - 1} \circ \tau^{m_{2 j_\ell - 1}}, h_{2 j_\ell - 1}) + d(h_{2 j_\ell - 1}, \hat{f}) \\
&+  \sum_{j=2 j_\ell}^{\infty} d(f \circ \tau^{m_{2 j_\ell} - 1}, 0) \\
\leq & 2^{-j\ell} + 2^{-j\ell} + d(h_{2 j_\ell - 1}, \hat{f}) + 2^{-j\ell} \to 0
\end{align*}
Analogous estimates hold for the other three cases.
\end{proof}

\begin{theorem} \hfill
\label{thm:twoauteukl}
\begin{enumerate}
\item For any translation $\tau \not\equiv 0$ of $\CC^n, n \geq 2,$ there exists a holomorphic automorphism $F$ of $\CC^n$ which is conjugate to a twisted shear map such that $\left< \tau, F  \right>$ is a free and dense (in compact-open topology) subgroup of the shear group of $\CC^n$.
\item For any translation $\tau \not\equiv 0$ of $\CC^n, n \geq 2,$ there exists a holomorphic automorphism $f$ of $\CC^n$ which is conjugate to a twisted overshear map such that $\left< \tau, F  \right>$ is a free and dense (in compact-open topology) subgroup of the overshear group of $\CC^n$.
\end{enumerate}
\end{theorem}

\begin{proof}
After conjugating with a special complex-linear map, we may assume that
\[
\tau(z) = \tau(z_1, \dots, z_n) = (z_1 + b, \dots, z_n + b), \quad b \in \RR_{>0}
\]
We then define the following elements of the overshear group of $\CC^n$:
\begin{align*}
I(z)        := & (z_2, \dots, z_n, -(-1)^n z_1) \\
F_{f,g}(z)  := & (z_2, \dots, z_n, \\ 
& -(-1)^n \exp(f(z_2,\dots, z_n)) \cdot z_1 + g(z_2,\dots, z_n) + (1 - (-1)^n) z_n) 
\end{align*}
where $f, g \in \holalg{\CC^{n-1}}$. Note that for $f \equiv 0$, these two elements actually belong to the shear group. The different signs depending on the dimension ensure that the Jacobian is equal to $1$.
\begin{enumerate}
\item We first treat the shear group, where $f \equiv 0$. By Theorem \ref{thm:pseudoconv} we may choose $g$ to be a hypercyclic element for the composition operator associated to the translation $(z_2, \dots, z_{n}) \mapsto (z_2 + b, \dots, z_{n} + b)$ of $\CC^{n-1}$.
For the conjugates of $F_{0,g}$ by the $m$-th iteration of $\tau$ we obtain:
\begin{align*}
& (\tau^{-m} \circ F_{0,g} \circ \tau^m)(z_1, \dots, z_n) = \\
& \qquad (z_2, \dots z_n, -(-1)^n z_1 + g(z_2 + m b, \dots, z_n + m b) + 2 z_n)
\end{align*}
Thus, by hypercyclicity of the composition operator we obtain that the set
$\{ \tau^{-m} \circ F_{0,g} \circ \tau^m, \; m \in \NN \}$ is dense in $\{ F_{0,h}, \; h \in \holalg{\CC^{n-1}} \}$. In particular it is possible to approximate the map $I$ by compositions of $F_{0,g}$ and $\tau$, hence also $I^n = -(-1)^n \id$ and $I^{-1}(z) = I^{2n-1}(z) = (-(-1)^n z_n, z_1, \dots, z_{n-1})$. Conjugating with $I^1, \dots, I^{n-1}$ will give twisted shears in all directions, and $F_{0, h} \circ I^{-1}$ is an ordinary shear. Therefore all shears and their compositions can be approximated this way.
\item For the overshear group we proceed in exactly the same way. By Proposition \ref{prop:doublehypercyc}, using an ordinary translation $\tau : \CC^n \to \CC^n$ it is possible to approximate all twisted overshears of the form $F_{0,\hat{g}}$ and $F_{\hat{f}, 0}$ and hence also $F_{\hat{f}, \hat{g}} = F_{0,\hat{g}} \circ I^{-1} \circ F_{\hat{f}, 0}$. 
\end{enumerate}
To ensure that the generated group is actually a free group, we need to show that no reduced word formed by $F_{f,g}$ and $\tau$ can equal the identity. This follows from an application of Nevanlinna theory similar to a degree argument as in Ahern and Rudin \cite{AhernRudin}. \qedhere
\end{proof}

\begin{corollary} \hfill
\begin{enumerate}
\item For any translation $\tau \not\equiv 0$ of $\CC^n, n \geq 2,$ there exists an $F \in \autvol{\CC^n}{\omega}$ such that $\left< \tau, F  \right>$ is a free and dense (in compact-open topology) subgroup of $\autvol{\CC^n}{\omega}$, where $\omega = dz_1 \wedge \dots \wedge dz_n$.
\item For any translation $\tau \not\equiv 0$ of $\CC^n, n \geq 2,$ there exists an $F \in \aut{\CC^n}$ such that $\left< \tau, F  \right>$ is a free and dense (in compact-open topology) subgroup of $\aut{\CC^n}$.
\end{enumerate}
\end{corollary}

\begin{proof}
By a result of Anders\'en \cite{Andersen}, the group of holomorphic shears is dense in $\autvol{\CC^n}{\omega}, \; \omega = dz_1 \wedge \dots \wedge dz_n$, and by a the corresponding result of Anders\'en--Lempert \cite{AndersenLempert}, the group of holomorphic overshears is dense in $\aut{\CC^n}$.
\end{proof}

\section{Density Property}
\label{sec:densprop}

The density property was introduced in Complex Geometry by Varolin \cite{Varolin1}, \cite{Varolin2}. For a survey about the current state of research related to density property and Anders{\'e}n--Lempert theory, we refer to Kaliman and Kutzschebauch \cite{KalimanKutzschebauch}.

\begin{definition}
\label{def:density}
A complex manifold $X$ has the \emph{density property} if in the compact-open topology the Lie algebra $Lie_{hol}(X)$ generated by completely integrable holomorphic vector fields on $X$ is dense in the Lie algebra $VF_{hol}(X)$ of all holomorphic vector fields on $X$.
\end{definition}

\begin{definition}
\label{def:densityvol}
Let a complex manifold $X$ be equipped with a holomorphic volume form $\omega$. We say that $X$ has the \emph{volume density property} with respect to $\omega$ if in the compact-open topology the Lie algebra $Lie^\omega_{hol}(X)$ generated by completely integrable holomorphic vector fields $\nu$ such that $\nu(\omega) = 0$, is dense in the Lie algebra $VF^\omega_{hol}(X)$ of all holomorphic vector fields that annihilate $\omega$.
\end{definition}

The following theorem is the central result of Anders{\'e}n--Lempert theory (originating from works of Anders\'en and Lempert \cite{Andersen}, \cite{AndersenLempert}), has been stated by Varolin \cite{Varolin1, Varolin2} after introducing the density property and is given in the following form in \cite{KalimanKutzschebauch} by Kaliman and Kutzschebauch, but essentially (for $\CC^n$) proven already in \cite{ForstnericRosay} by Forstneri\v{c} and Rosay.
\begin{theorem}[Theorem 2 in \cite{KalimanKutzschebauch}]
\label{thm:andlemp}
Let $X$ be a Stein manifold with the density (resp. volume density) property and let $\Omega$ be an open subset of $X$. In case of volume density property further assume that in de-Rham cohomology $H^{n-1}(\Omega,\CC) = 0$. Suppose that $\Phi : [0, 1] \times \Omega \to X$ is a $\cont^1$-smooth map such that
\begin{enumerate}
\item $\Phi_t : \Omega \to X$ is holomorphic and injective (and resp. volume preserving) for every $t \in [0, 1]$
\item $\Phi_0 : \Omega \to X$ is the natural embedding of $\Omega$ into $X$
\item $\Phi_t(\Omega)$ is a Runge subset of $X$ for every $t \in [0, 1]$
\end{enumerate}
Then for each $\varepsilon > 0$ and every compact subset $K \subset \Omega$ there is a continuous family
$\alpha : [0, 1] \to \aut{X}$ of holomorphic (and resp. volume preserving) automorphisms of $X$ such that
\[
\alpha_0 = \id \; \mbox{and} \; \sup_K d(\alpha_t,\Phi_t) < \varepsilon
\]
for every $t \in [0, 1]$.
\end{theorem}

\begin{examples} \hfill 
\begin{enumerate}
\item \label{dens-ex1} $\CC^n, \; n \geq 2,$ have the density property.
\item \label{dens-ex2} $\CC^* \times \CC^*$ has the volume density property for the holomorphic volume form $\omega = \frac{\mathrm{d} x}{x} \wedge \frac{\mathrm{d} y}{y}$. Whether it has the density property is not clear.
\item Homogeneous Stein manifolds $X = G/K$, where $G$ is a semi-simple Lie group have the density property. (see \cite{dens3}).
\item Linear algebraic groups except those with $\CC$ and $(\CC^*)^n$ as a connected component have the density property, and all linear algebraic groups have the volume density property with respect to the Haar form (see \cite{dens4-1}, \cite{KalimanKutzschebauch1}). Note that examples \ref{dens-ex1} and \ref{dens-ex2} are special cases of linear algebraic groups.
\item A hypersurface $H \subset \CC^{n+2}$ of the form $H = \{ (x, u, v) \in \CC^n \times \CC \times \CC \,:\, f(x) = u \cdot v\}$ where $f : \CC^n \to \CC$ is a polynomial with smooth zero fiber, has both the density property and the volume density property with respect to a unique algebraic volume form (see \cite{dens4-1}, \cite{KalimanKutzschebauch2}).
\end{enumerate}
\end{examples}

We will also need the following lemma which makes the distinction between connected and path-connected component superfluous for groups of holomorphic automorphisms.

\begin{lemma}
\label{lem:autoconnected}
The connected component of the group of holomorphic automorphisms resp. volume-preserving holomorphic automorphisms of a complex manifold resp. complex manifold with a holomorphic volume form is $\cont^1$-path-connected.
\end{lemma}
\begin{proof}
See Lind \cite[Remark 6.6]{Lind}. 
\end{proof}

\begin{theorem}
\label{thm:twoautdens}
Let $X$ be a Stein manifold with density property and assume there exists a generalized translation $\tau \in \autid{X}$. Then there exists an $F \in \autid{X}$ such that $\left< \tau, F  \right>$ is a free and dense (in compact-open topology) subgroup of $\autid{X}$.
\end{theorem}

\begin{proof}
Because $X$ is Stein, there exists an exhaustion $\{L_j\}_{j \in \NN}$ of $X$ with $\holalg{X}$-convex compacts.

We choose a countable dense subset (in compact-open topology) $\{ g_j \}_{j \in \NN} \subset \autid{X}$. Such a set exists, since we can think of $\autid{X} \subset \holalg{X, X} \subset \holalg{\CC^N, \CC^N}$ as metric spaces, where we take $X$ to be properly embedded into $\CC^N$, and $\holalg{\CC^N, \CC^N}$ is known to be a separable Fr\'echet space.

We now construct $F$ inductively as a sequence of maps $F_j, \; j \in \NN, \; $ together with sequences $\{\varepsilon_j\}_{j \in \NN} \subset \RR_{> 0}$ with $\sum_{j=1}^\infty \varepsilon_j < \infty$, $\{k(j)\}_{j\in\mathbb N}, k(j)\in\mathbb N, k(j)<k(j+1)$, and $m_j\in\mathbb N$, such that the following hold at step $j$:

\begin{itemize}
\item[($a_j$)] $\sup_{L_{k(i)}} d(\tau^{-m_i} \circ F_j \circ \tau^{m_i}, g_i) < \varepsilon_i$, $\forall \, i \leq j$, 
\item[($b_j$)] $\sup_{L_{k(j)}}\left( d(F_j, F_{j-1}) + d(F_j^{-1},F^{-1}_{j-1}) \right) < \epsilon_j$. 
\end{itemize}

\medskip

We set $F_1=Id, \epsilon_1=k(1)=m(1)=1$.  Assuming that $g_1=\id$ we have ($a_1$) and ($a_2$) is void.   

\medskip

Assume now that we have constructed $F_{j-1}$ along with the integers for some $j\geq 2$.   

\medskip

Choose $k(j)$ so large such that 
\begin{equation*}\label{eq:condition}
\tau^{m_i}(L_{k(i)})\subset L^\circ_{k(j)} \mbox{ for all } i<j.  
\end{equation*}

\medskip

By Lemma \ref{lem:autoconnected} there exist  $\cont^1$-paths $[0,1] \ni t \to \varphi^j_t, \psi^j_t$ in $\autid{X}$ connecting
$F_{j-1}$ and $g_j$  respectively to the identity.    Let $K_j$ be a large enough holomorphically convex compact set 
containing both the complete $\varphi^j_t$-orbit of $L_{k(j)}\cup F_{j-1}^{-1}(L_{k_j})$ and the complete $\psi^j_t$-orbit of $L_{k(j)}$ 
in its interior.   Choose $m_j$ large enough such that $K_j$ and $\tau^{m_j}(K_j)$ are disjoint and such that 
their union is $\holalg{X}$-convex.   Let $C_j:=\widehat{[L_{k(j)}\cup F_{j-1}^{-1}(L_{k(j)})]}$.
We define an isotopy $\Phi^j_t$ of maps near $C_j\cup\tau^{m_j}(L_{k(j)})$ by setting it equal to $\varphi^j_t$ near $C_j$
and $\tau^{m_j}\circ\psi^j_t\circ\tau^{-m_j}$ near $\tau^{m_j}(L_{k(j)})$.

\medskip

For any $\tilde\epsilon_j$ we may now apply the Anders\'en--Lempert Theorem and obtain an $F_{j} \in \autid{X}$ such that
\[
\sup_{L_{k(j)}} d(F_j, F_{j-1})+d(F_j^{-1},F_{j-1}^{-1}) < \tilde\varepsilon_j\]

and 
\[
\sup_{\tau^{m_j}(L_{k(j)})} d(F_j, \tau^{m_j}\circ g_j\circ\tau^{-m_j}) < \varepsilon_j
\]

\medskip

We see that we may choose $\tilde\epsilon_j<2^{-j}$ small enough such that $(a_j)$ and $(b_j)$ are satisfied.
\medskip

To ensure that the group generated by $\tau$ and $F$ is free, we need to avoid that there is a non-trivial reduced word (formed by $\tau$ and $F$) equal to the identity. For induction step $j$ there are finitely many words of length $\leq j$. By changing the choice of $f_j$ within the required bounds we can make sure that on the compact $L_{k(j)}$ no word of length $\leq j$ equals the identity. Set $\delta_j :=\min \{ \sup_{L_{k(j)}} d(w, \id) \; : \; w \neq \id \mbox{ reduced word of length } \leq j, \mbox{ consisting of } \tau \mbox{ and } F_j \}$ and choose all subsequent $\varepsilon_k, k \geq j + 1,$ small enough such that $\sum_{k=j+1}^\infty C_k \varepsilon_k < \delta_j$ where the constants $C_k$ take into account the number of possible words.

\medskip

We have now obtained a sequence $F_{j}$ of holomorphic automorphisms of $X$ and it follows from the conditions $(b_j)$
the $F_j\rightarrow F\in\autid{X}$.
It is immediate from the condition $(a_j)$ that $\{ \tau^{-m_j} \circ F_j \circ \tau^{m_j} \}_{j \in \NN}$ is dense in $\autid{X}$.
It follows that $ \tau^{-m_j} \circ F \circ \tau^{m_j} =  (\tau^{-m_j} \circ F_j \circ \tau^{m_j}) \circ (\tau^{-m_j} \circ (F_j^{-1} \circ F) \circ \tau^{m_j}), j \in \NN,$ s dense sequence in $\autid{X}$, since $\tau^{-m_j} \circ (F_j^{-1} \circ F) \circ \tau^{m_j} \to \id$ provided $\varepsilon_j$ is decreasing fast enough.
\end{proof}

\begin{remark}
In case $(X, \omega)$ would be a Stein manifold of dimension $n \in \NN$ with volume density property and an exhaustion of compacts $\{L_j\}_{J \in \NN}$ with de Rham cohomology $H^{n-1}(L_j, \CC) = 0$, one has an similar statement as in Theorem \ref{thm:twoautdens}. Actually, a careful investigation of the proof of the Anders\'en--Lempert Theorem shows that the condition $H^{n-1}(L_j, \CC) = 0$ can be dropped for components of $L_j$ where one wants to approximate the identity and for components where one wants to approximate an already globally defined automorphism, provided that $H^{n-1}(X, \CC) = 0$ holds.
\end{remark}

\begin{examples} In the following examples of Stein manifolds with volume density property, the method proof of the preceding theorem fails for various reasons:
\begin{itemize}
\item The manifold $\CC^*$ with volume form $\omega = \frac{d z}{z}$ has the volume density property; $\aut{\CC^*} = \autvol{\CC^*}{\omega} = \{ z \mapsto a \cdot z^{\pm 1} \; : \; a \in \CC^*\}$. Therefore it is obvious that no generalized translation exists on $\CC^*$. In addition there is due to dimensional reasons ($n-1=0$) the obstruction $\dim H^{0}(\CC^*, \CC) = 1$.
\item The manifold $\CC$ with volume form $\omega = d z$ has the volume density property as well, and actually all volume-preserving automorphisms consist of translations. However again due to dimensional reasons there is the obstruction $\dim H^{0}(\CC^*, \CC) = 1$.
\item The manifolds $X := \CC^* \times \CC^*$ and  $Y := \CC \times \CC^*$  with volume forms $\omega_X = \frac{d z}{z} \wedge \frac{d w}{w}$ resp. $\omega_Y = d z \wedge \frac{d w}{w}$  are known to have the the volume density property. Generalized translations, e.g. $\tau_X(z,w) = (2z, \frac{1}{2}w), \; \tau_Y(z,w) = (z + 1,w)$, exist, but in these cases ($n-1=1$) there is a topological obstruction $\dim H^{1}(X, \CC) = 2$ resp. $\dim H^{1}(Y, \CC) = 1$.
\end{itemize}
\end{examples}

The proof of Theorem \ref{thm:twoautdens} reveals as well that:
\begin{theorem}
\label{thm:densehyperc}
Let $X$ be a Stein manifold with density property and assume there exists a generalized translation $\tau \in \autid{X}$. Then the associated conjugation operator $\widetilde{C}_\tau$, defined by $\widetilde{C}_\tau(f) = \tau \circ f \circ \tau^{-1}$, is hypercyclic on $\autid{X}$.
\end{theorem}

\section{Existence of Generalized Translations}

In the last section we give a couple of examples of manifolds with Density Property with a generalized translation.

\begin{example}
Let $Y$ be a Stein manifold, then there exists a generalized translation on $X = \CC \times Y$, given by $(z, y) \mapsto (z + a, y)$ for $a \in \CC^\ast$. In case $Y$ is a complex Lie group, $Y$ has the Density Property (Varolin \cite{Varolin1}).
\end{example}

\begin{example}
Let $Y$ be a Stein manifold, then there exists a generalized translation on $X = \CC^\ast \times \CC^\ast \times Y$, given by $(z, w, y) \mapsto (a z, w / a, y)$. In case $Y$ is a complex Lie group with Volume Density Property, $Y$ has the Volume Density Property (Varolin \cite{Varolin1}).
\end{example}

\begin{example}
Let $p(z) \in \CC[z]$ be a polynomial with simple roots. Then the so-called \emph{Danielewsi surface}
\[
D_p := \{ (x,y,z) \in \CC^3 \; : \; x \cdot y = p(z) \}
\]
has the Density Property (Kaliman and Kutzschebauch \cite{KalimanKutzschebauch2}). The following map is a generalized translation $\tau_a, \; a \in \CC^\ast,$ on $D_p$:
\begin{align*}
(x, y, z) &\mapsto \left(x, \frac{p(z + a \cdot x)}{x}, z + a \cdot x\right) \\
          &= \left(x, y + \frac{p(z + a \cdot x) - p(z)}{x}, z + a \cdot x\right) \\
          &= \left(x, y + \frac{p(z + a \cdot x) - p(z)}{x}, z + a \cdot x\right) \\
          &= \left(x, y + p^{\prime}(z) \cdot a + \sum_{k=2}^{\infty} p^{(k)}(z) \cdot a^k \cdot x^{k-1}, z + a \cdot x\right) \\
\end{align*}
Note that any compact of $D_p$ can be moved as far away as possible by iterations of $\tau^m_a = \tau_{a m}$. The only crucial case one has to check is for $x = 0$: this implies $p(z) = 0$; but since $p$ has only simple and finitely many zeros, $p^\prime(z)$ is bounded away from $0$, and so is the shift in $y$-direction.
\end{example}

\bibliographystyle{amsplain}

\end{document}